\documentclass[12pt]{article}
\usepackage{amsfonts}
\usepackage{amsthm,latexsym}
\usepackage{amsmath}
\usepackage[mathscr]{eucal}

\DeclareMathOperator{\Hom}{Hom}

\theoremstyle{plain}
\newtheorem{lemma}{Lemma}[section]
\newtheorem{theorem}[lemma]{Theorem}
\newtheorem{prop}[lemma]{Proposition}
\newtheorem{corollary}[lemma]{Corollary}
\theoremstyle{definition}
\newtheorem{example}[lemma]{Example}
\newtheorem{definition}[lemma]{Definition}

\def\d{$\displaystyle}
\def\be{\begin{equation}}

\begin{document}

\title{Reciprocity laws for representations of finite groups}

\author
{Sunil Chebolu \\ schebol@ilstu.edu \\
Dept. of Mathematics, Illinois State University, Normal, IL, USA. \footnote{Supported in part by NFIG (New Faculty Initiative Grant) at ISU}
\and
J\'{a}n Min\'{a}\v{c} \\ minac@uwo.ca  \\
Dept. of Mathematics, Univ. of Western Ontario, London, ON, Canada.\footnote{Supported in part by NSERC grant \#0370A01}
\and
Clive Reis \\ c.m.reis@shaw.ca \\
222 B Saint Charles St., Victoria, BC, Canada.
}


\date{\emph{Dedicated to Professor Paulo Ribenboim who inspired us to
read the masters and to dream.}
 }
\maketitle


\begin{abstract}

Much has been written on reciprocity laws in number theory and their
connections with group representations. In this paper we explore more on these
connections.  We prove a ``reciprocity Law" for certain specific
representations of semidirect products of two cyclic groups which is in
complete analogy with classical reciprocity laws in number theory.  In fact,
we show that the celebrated quadratic reciprocity law is a direct consequence
of our main theorem applied to a specific group.   As another consequence of
our main theorem we also recover a classical theorem of Sylvester. Our main
focus is on explicit constructions of representations over  sufficiently small
fields. These investigations give further evidence that there is still much
unexplored territory in connections between number theory and group
representations, even at an elementary level.
\end{abstract}

\section{Introduction}

Let $p$ and $s$ be odd primes.  Then the quadratic reciprocity law tells
us how to find all finite fields $\mathbb{F}_s$ of $s$ elements for which
$\sqrt{p} \in \mathbb{F}_s$.  (To anticipate the generalization we have in
mind, we might say that $\sqrt{p}$ is {\it realizable} over $\mathbb{F}_s$.)
Remarkably, whether or not $\sqrt{p}$ is realizable over $\mathbb{F}_s$ can
be decided mod $4p$ thus reducing a question concerning infinitely many
fields to one which can be decided using only finitely many operations,
namely, squarings.  (See [Gau] or [Ser].)

It is known that reciprocity laws are intimately connected with
representations of finite groups (see for example [Art], [Lan], [Tat].)
However, we were unable to find in the literature a development of reciprocity
laws for representations of finite groups themselves. In a recent preprint
\cite{BH} an extremely interesting connection between division algebras with
involutions and automorphism groups associated with Shimura varieties is
uncovered and used. Certain constructions employed in this paper have a
similar flavor as in  \cite{BH}. Also in \cite{Lem} the field of
definition of some representations of finite groups was studied in
order to deduce statements about the ranks of class groups.

The main goal of this paper is to provide reciprocity laws for certain
representations of a restricted family of metacyclic groups.  More
precisely, we define certain specific representations $\rho_s : G\to Aut
V_s$ where $G$ is a fixed metacyclic group and $V_s$ is a vector space
over $\overline{\mathbb{F}}_s$, the algebraic closure of $\mathbb{F}_s$.  We
then show that the question of the realizability of $\rho_s$ over a given
finite extension field of $\mathbb{F}_s$ can be decided entirely in terms of
the invariants of $G$ and depends only on a finite number of
computations.  This provides a straightforward analogy to the classical
reciprocity laws.  Indeed, in section 3 below, it will be seen that a
judicious choice of group $G$ and corresponding representation
$\rho(G)$ yields the usual quadratic reciprocity law.

The connection between the representation $\rho(G)$ and the quadratic
reciprocity law was discovered by D. R. Corro (see [Jac], pp. 320-325).
However this is used only to evaluate the square of Gauss's sum
which is only part of the proof.  To complete that part of the proof in
which the reduction to a finite number of primes is achieved, a
classical method due to Jacobi is used.  This part is not difficult and
is actually worked out in a simple way in [Ser 2].  However the
evaluation of the square of Gauss's sum though important, does not, on
its own, accomplish the reduction from infinitely many to finitely many
primes.

Some of the results we obtain could be arrived at using the notion of
the Schur index of a representation which is always $1$ when the field
in question is finite and of characteristic coprime to $|G|$ (see
[Dor]).  However our approach has the merit of being quite elementary
and, more importantly, constructive.  Except for section 4, the basic
notions of representation theory as can be found in [Ser 1] are more
than enough.  For section 4 we refer the reader to [Rei].  All other
background material concerning finite fields, Vandermonde and companion
matrices and characteristic polynomials should be understandable to a
good advanced undergraduate.

In this paper we believe that we have merely scratched the surface of a
possibly rather general theory of reciprocity in the representations of
finite groups. The style of this paper is influenced by Paulo Ribenboim's writing, and after his urging to read Euler, also writing of Euler which are full of examples,  ``naive questions'' and exploration spirit.

\section{Main Results -- Galois, Vandermonde and field of definition}

Throughout we shall be concerned with split semidirect products of two cyclic
groups.  Thus in terms of generators and relations $G = \langle a,b \mid
a^m = 1 = b^n; b^{-1}ab = a^k\rangle$, where the order of $k$ (mod $m$) divides $n$.  Let $s$ be an odd prime
relatively prime to $m$ and let $\zeta$ be a primitive $m^{th}$ root of
unity in $\overline{\mathbb{F}}_s$, the algebraic closure of the field
$\mathbb{F}_s$ of $s$ elements.  We shall focus attention on the
representation $\rho^G$ induced from the representation $\rho :\langle
a\rangle \to \overline{\mathbb{F}}_s^*$ defined by $\rho(a) = \zeta$.
We first prove a simple proposition which, among other things, tells us
under what circumstances $\rho^G$ is irreducible over
$\overline{\mathbb{F}}_s$.

\begin{prop}
(1)  Let $G,\rho$ and $\rho^G$ be as above.  Then $\rho^G$ is
irreducible over $\overline{\mathbb{F}}_s$ iff $|k|_m$, the order
of $k$ in the multiplicative group of units mod $m$, is equal to the
order of $b$.\\
(2)  If $|k|_m = t\not= n, n = tr$ and $(r,s) = 1$, then
$\rho^G\approx\rho_0\dot{+}\rho_1\dot{+}...\dot{+}\rho_{r-1}$ where the
$\rho_i$ are irreducible pairwise inequivalent representations over
$\overline{\mathbb{F}}_s$ and, relative to an appropriate basis
$\mathscr{B}_i$,
\[ [\rho_i(a)]_{\mathscr{B}_i} = \left( \begin{array}{cccc}
\zeta & & & 0 \\
& \zeta^k \\
& & \ddots & \\
0 & & & \zeta^{k^{t-1}} \end{array}\right);\quad [\rho_i(b)]_{\mathscr{B}_i} =
\left( \begin{array}{cccccc}
0 & 0 & 0 & \cdots & 0 & \eta_i^{-1} \\
1 & 0 & 0 & \cdots & 0 & 0 \\
0 & 1 & 0 & \cdots & 0 & 0 \\
\vdots & \vdots & \vdots & & \vdots & \vdots \\
0 & 0 & 0 & \cdots & 1 & 0  \end{array} \right) \]
Here $\eta_i = \eta^i, i = 0,1,\dots, r-1$ where $\eta$ is a primitive
$r^{th}$ root of unity in $\overline{\mathbb{F}}_s$.
\end{prop}

\begin{proof}
Since $a = b^{-n}ab^n = a^{k^n}$, it follows that $k^n\equiv 1(m)$ and so
$|k|_m\mid n$.  Let $M = \overline{\mathbb{F}}_s$ be the
representation space of $\rho$ and let $H = \langle a\rangle$.  Then the
representation space of $\rho^G$ is $V = \overline{\mathbb{F}}_s
G\otimes_{\overline{\mathbb{F}}_sH}M$.  The set $\{b^i\otimes 1\mid i =
0,1,..., n-1\}$ is a basis of $V$. Setting $e_i = b^i\otimes 1, i =
0,..., n-1$, we have $be_i = e_{i+1}$ where the indices are taken
modulo $n$, and $ae_i = ab^i\otimes 1 = b^i\otimes(a^{k^{i}}\cdot 1) =
\zeta^{k^i}{e_i}$.

We prove the sufficiency of (1) first.  Assume therefore that
$|k|_m = n$.  Then $\{\zeta,\zeta^k,..., \zeta^{k^{n-1}}\}$ is a set of
$n$ distinct elements.  Let $\zeta_i = \zeta^{k^{i}}$ and let $W$ be a
nonzero $G$-invariant subspace of $V$.  Let $f = c_ve_v + \cdots +
c_{n-1}e_{n-1}, c_v \not= 0$ be a nonzero vector of $W$ such that $v$ is
largest subject to $c_v\not= 0$ and $c_0 = c_1 = \cdots = c_{v-1} =
0$.  Then $c_{n-1}\not= 0$, otherwise $bf = c_ve_{v+1} +\cdots +
c_{n-2}e_{n-1}\in W - \{0\}$, contradicting maximality of $v$.  Now
\d \zeta_vf = \sum_{j=v}^{n-1} \zeta_vc_je_j$ and \d af =
\sum_{j=v}^{n-1} \zeta_jc_je_j$ and assume $v < n-1$.  Then \d af -
\zeta_vf = \sum_{j=v+1}^{n-1} c_j(\zeta_j - \zeta_v)e_j$ is a nonzero
vector of $W$ since $c_{n-1}(\zeta_{n-1} - \zeta_v)\not= 0$,
contradicting maximality of $v$.  Thus $v = n - 1$ and so $e_{n-1}\in
W$.  It follows that $b^ie_{n-1}\in W$ for all $i$ and so $W = V$,
proving the irreducibility of $V$.

Next we prove (2) and indicate that the existence of one of the direct
summands constructed does not depend on the existence of a primitive
$r^{th}$ root of unity in $\overline{\mathbb{F}}_s$.  This will then
also prove the necessity of (1).

Let $\mathscr{B} = \{e_0,e_1,..., e_{n-1}\}$ where the $e_i$ are defined
above.  Then
\[ [\rho^G(a)]_\mathscr{B} = \left( \begin{array}{cccc}
A & & & 0 \\
& A & &  \\
& & \ddots & \\
0 & & & A \end{array}\right) \]

where \[ A = \left( \begin{array}{cccc}
\zeta_0 & & & 0 \\
& \zeta_1 & & \\
& & \ddots & \\
0 & & & \zeta_{t-1} \end{array}\right) \mbox{ and there are $r A$'s;} \]

\[ [\rho^G(b)]_\mathscr{B} =
\left( \begin{array}{ccccc}
0 & 0 & \cdots & 0 & 1 \\
1 & 0 & \cdots & 0 & 0 \\
0 & 1 & \cdots & \vdots & \vdots \\
\vdots & \vdots & & 0 & 0 \\
0 & 0 & \cdots & 1 & 0  \end{array} \right) \quad . \]

Let $\eta$ be a primitive $r^{th}$ root of unity in
$\overline{\mathbb{F}}_s$ and let \d v_{ij} = \sum^{r-1}_{\ell =0}
\eta^{i\ell}e_{j+\ell t}, 0\leq i\leq r-1, 0\leq j\leq t-1$ and set
\[ W_i = {\it span}\;\mathscr{B}_i, \quad \mathscr{B}_i = \{v_{i0},v_{i1},\dots
, v_{i,t-1}\}\; . \]

\noindent Now \quad \d \begin{array}[t]{l} av_{ij} = \zeta_jv_{ij}\; {\rm while} \\
bv_{ij} = v_{i,j+1}\quad{\rm if}\quad 0\leq j\leq t - 2 \\
\end{array}$ \\ and
\begin{eqnarray*}
bv_{i,t-1} & = & \sum^{r-1}_{\ell =0} \eta^{i\ell}be_{t-1+\ell t} \\
& = & \sum^{r-1}_{\ell =0} \eta^{i\ell}e_{t(\ell +1)} \\
& = & \eta^{-i}\sum^{r-1}_{\ell =0} \eta^{i(\ell +1)}e_{t(\ell +1)} \\
& = & \eta^{-i}v_{i0}\; .
\end{eqnarray*}

It follows that $W_i$ is a $G$-invariant subspace of $V$ and the matrix
representation of the restriction $\rho_i$ of $\rho^G$ to $W_i$ relative
to the basis $\mathscr{B}_i$ is given by:
\[ [\rho_i(a))]_{\mathscr{B}_i} = \left( \begin{array}{cccc}
\zeta_0 & & & 0 \\
& \zeta_1 & & \\
& & \ddots & \\
0 & & & \zeta_{t-1} \end{array}\right) \]

\[ [\rho_i(b)]_{\mathscr{B}_i} =
\left( \begin{array}{ccccc}
0 & 0 & \cdots & 0 & \eta^{-i} \\
1 & 0 & \cdots & 0 & 0 \\
0 & 1 & \cdots & \vdots & \vdots \\
\vdots & \vdots & & 0 & 0 \\
0 & 0 & \cdots & 1 & 0  \end{array} \right) \quad . \]

There are $n$ vectors in $\cup \mathscr{B}_i$ and so, if we can prove
they are linearly independent we shall have the decomposition $V = W_0
\oplus W_1 \oplus \cdots \oplus W_{r-1}$ into $G$-spaces.

Let \d \sum_{i,j} \alpha_{ij}v_{ij}$.  Then

\[ \sum_{i,j,\ell} \alpha_{ij}\eta^{i\ell}e_{j+\ell t} = 0\;. \]

Fix $j$ and $\ell$.  We have

\[ \sum^{r-1}_{i=0} \alpha_{ij}(\eta^\ell)^i = 0 \quad \mbox{for all }
j,\ell \; . \]

Fixing $j$ and letting $\ell$ vary, we obtain a system of $r$ equations
in $r$ unknowns with matrix of coefficients the Vandermonde matrix

\[ \left( \begin{array}{cccc}
1 & 1 & \cdots & 1 \\
1 & \eta & \cdots & \eta^{r-1} \\
1 & \eta^2 & \cdots & \eta^{2(r-1)} \\
\vdots & \vdots & & \vdots \\
1 & \eta^{r-1} & \cdots & \eta^{(r-1)(r-1)}  \end{array} \right) \quad . \]

Since $\eta$ is a primitive $r^{th}$ root of unity, this matrix is
nonsingular whence $\alpha_{ij} = 0$ for all $i$.  Letting $j$ vary we
obtain $\alpha_{ij} = 0$ for all $i$ and $j$.  Observe that the
$\overline{\mathbb{F}}_sG$-module $W_0$ exists whether or not $(r,s) = 1$
and so the necessity of (1) is proved.

The irreducibility of each $W_i$ is proved in a manner entirely similar
to that used to prove the sufficiency of (1).

Finally, we observe that the characteristic polynomial of $\rho_i(b)$ is
$X^t - \eta^{-i}$ whence the elements $\rho_0,\rho_1,\dots \rho_{r-1}$ are mutually
inequivalent.
\end{proof}
Using the notation established in the foregoing, we now proceed by a
series of lemmas to prove our main theorem.

\begin{lemma}
Let $\zeta$ be a primitive $m^{th}$ root of unity in $\overline{\mathbb{F}}_s$
and let $k$ be a positive integer with $|k|_m = n$.  Then the Frobenius
automorphism $\tau$ on $\overline{\mathbb{F}}_s$ defined by $\tau (x) = x^q$
where $q$ is a power of $s$ permutes the elements of $\mathscr{P} = \{\zeta
,\zeta^k,\dots ,\zeta^{k^{n-1}}\}$ iff $q\equiv k^i(m)$ for some $i, 0\leq
i\leq n-1$.
\end{lemma}

\begin{proof}
If $\tau$ permutes the elements of $\mathscr{P}$, then
$\zeta^q\in\mathscr{P}$ and so $\zeta^q = \zeta^{k^i}$ for some $i,\;
0\leq i\leq n-1$.  Hence $q\equiv k^i(m)$.

Conversely, if $q\equiv k^i(m)$ then $(\zeta^{k^j})^q =
\zeta^{k^{i+j}}\in \mathscr{P}$ whence $\tau$ permutes the elements of
$\mathscr{P}$.
\end{proof}

The following is the cornerstone of our main result:

\begin{lemma}
Let $K$ be a field and let $\alpha$ be an automorphism of $K$ with fixed
field $F$.  Let $a_0,a_1,\dots, a_{n-1}$ be elements of $K$ which are
cyclically permuted by $\alpha$, say $\alpha (a_i) = a_{i+1}$ where the
indices are taken modulo $n$.  Let \d f(x) = \prod^{n-1}_{i = 0} (x-a_i)\in
F[x]$ and let $C = (c_{ij})$ be the $n\times n$ matrix with $c_{i,i-1} =
1$ if $2\leq i\leq n$; $c_{1n} = 1$ and $c_{ij} = 0$ otherwise.  Let

\[ V =
\left( \begin{array}{ccccc}
1 & a_0 & a^2_0 & \cdots & a^{n-1}_0 \\
1 & a_1 & a^2_1 & \cdots & a^{n-1}_1 \\
\vdots & & & & \vdots \\
1 & a_{n-1} & a^2_{n-1} & \cdots & a^{n-1}_{n-1} \end{array}\right) \; ,
\mbox{\it the } n\times n \mbox{\it Vandermonde matrix.$\;$} \] Then
$V^{-1}CV\in  M_n(F)$. (Here $M_{n}(F)$ denotes the algebra of $n \times n$
matrices over $F$.)
\end{lemma}

\begin{proof}
Define $g_i(x) = \frac{f(x)}{(x-a_i)f'(a_i)}$, $i = 0,1,\dots n-1$ where
$f'(a_i)$ is the formal derivative of $f$ evaluated at $a_i$.  Let $g_i(x) =
d_{0i} + d_{1i}x +\cdots + d_{(n-1)i}x^{n-1}$ and let

\[ D =
\left( \begin{array}{llll}
d_{00} & d_{01} & \cdots & d_{0n-1} \\
d_{10} & d_{11} & \cdots & d_{1n-1} \\
\vdots & & \cdots & \vdots \\
d_{n-1,0} & d_{n-1,1} & \cdots & d_{n-1,n-1} \end{array}\right) \;
\mbox{, i.e., $D$ is the matrix whose columns}
\]
are the coefficients of the $g_i$.  Since $g_i(a_j) = \delta_{ij}$, the
Kronecker delta, we have $VD = I$ and so $D = V^{-1}$.  Now

\[ V^{-1}C =
\left( \begin{array}{lllll}
d_{01} & d_{02} & \cdots & d_{0n-1} & d_{00} \\
d_{11} & d_{12} & \cdots & d_{1n-1} & d_{10} \\
\vdots & \vdots & \cdots & \vdots & \vdots \\
d_{n-1,1} & d_{n-1,2} & \cdots & d_{n-1,n-1} & d_{n-1,0} \end{array}\right)
\]
and so it is easily seen that the $(j+1)^{st}$ column of $V^{-1}CV$
consists of the coefficients of the polynomial $h_j(x) = a^j_0g_1(x) +
a^j_1g_2(x) + \cdots + {a_{n-2}}^{j}g_{n-1}(x) + {a_{n-1}}^{j}g_0(x)$.  Now
$g_i(x) = \frac{f(x)}{(x-a_i)f'(a_i)}$ and so applying the automorphism
$\alpha$ and bearing in mind the fact that $f(x)\in F[x]$ we get, if
$0\leq i\leq n-2$,
\[ \alpha(g_i(x)) = \frac{f(x)}{(x-a_{i+1})f'(a_{i+1})} = g_{i+1}(x)
\mbox{  while} \]
\[\alpha (g_{n-1}(x)) = \frac{f(x)}{(x-a_0)f'(a_0)} = g_0(x)\; . \]
Thus $\alpha (h_j(x)) = a^j_1g_2(x)\cdots + {a_{n-1}}^{j}g_0(x) + a^j_0g_1(x)
= h_j(x)$.  Since the fixed field of $\alpha$ is $F$, it follows that
$V^{-1}CV\in M_n(F)$.
\end{proof}
\noindent
{\bf Remark.} The referee of our paper found the following nice, less
computational proof of Lemma~2.3.  Let $\beta=\alpha^{-1}$. Then
$\beta$ acts naturally on the matrices in $M_n(K)$, element-wise, and we have $\beta
(V)=CV$. Hence $\beta(V^{-1})=(\beta(V))^{-1}=(CV)^{-1}=V^{-1}C^{-1}$.
Thus $\beta(V^{-1}CV)=\beta(V^{-1})\beta(C)\beta(V)=V^{-1}C^{-1}CCV=
V^{-1}CV$ and again we can conclude that $V^{-1}CV\in M_n(F)$.

In the next lemma we use the hypothesis that the Frobenius automorphism
$\tau (x) = x^q,x\in \overline{\mathbb{F}}_s$ is transitive on the set
$\mathscr{P} = \{\zeta,\zeta^k,\dots,\zeta^{k^{n-1}}\}$. Observe that
this simply means that $q$ (mod $m$) and $k$ (mod $m$) generate the
same subgroups of $(\mathbb{Z}/m\mathbb{Z})^*$

\begin{lemma}
Let $G = \langle a,b\mid a^m = 1 = b^n; b^{-1}ab = a^k\rangle$ where
$|k|_m = n$.  Let $\zeta$ be a primitive $m^{th}$ root of unity in
$\overline{\mathbb{F}}_s$ and let $q$ be a power of $s$.  Assume that
the Frobenius automorphism $\tau$ on $\overline{\mathbb{F}}_s$ defined
by $\tau (x) = x^q$ is transitive on $\mathscr{P} = \{\zeta
,\zeta^k,\dots , \zeta^{k^{n-1}}\}$ and let $\rho :\langle
a\rangle\to\overline{\mathbb {F}}^*$ be the representation defined by
$\rho (a) = \zeta$.  Then the induced representation $\rho ^G$ is
realizable over $\mathbb{F}_q$.
\end{lemma}

\begin{proof}
By assumption, $\zeta ,\zeta^q,\dots , \zeta^{q^{n-1}}$ are distinct and
$\zeta^{q^n} = \zeta$ whence $|q| _m = n$.  By Lemma 2.2, $q\equiv k^i(m)$ for
some $i$.  Since $|q|_m = |k|_m$, it follows that $(i,n) = 1$.  Hence, letting
$c = b^i$ we have $G = \langle a,c\mid a^m = 1 = c^n$, $c^{-1}ac =
a^q\rangle$.  The induced representation $\rho ^G$ using the coset
representatives $1,c,c^2,\dots , c^{n-1}$ is given by

\[ \rho^G(c) =
\left( \begin{array}{ccccc}
0 & 0 & \cdots & 0 & 1 \\
1 & 0 & \cdots & 0 & 0 \\
0 & 1 & \cdots & 0 & 0 \\
\vdots & \vdots & \cdots & \vdots & \vdots \\
0 & 0 & \cdots & 1 & 0 \end{array}\right) \; ,\quad
\rho^G(a) = \left( \begin{array}{cccc}
\zeta & & & 0 \\
& \zeta^q & & \\
& & \ddots & \\
0 & & & \zeta^{q^{n-1}} \end{array}\right) \quad .
\]
Set $\rho^G(c) = C$, $\rho^G(a) = A$ and let \d f(x) = \prod^{n-1}_{i=0}
(x-\zeta^{q^i})$.  Since $\tau (f(x)) = f(x)$ it follows that $f(x)\in
\mathbb{F}_q[x]$, say $f(x) = a_0 + a_1x + \cdots + x^n$.

Let $W$ be an $n$-dimensional vector space over
$\overline{\mathbb{F}_s}$ and let $\mathscr{A} = \{v_0,v_1,v_2,\dots,
v_{n-1}\}$ be a basis of $W$.  Let $L:W\to W$ be the linear
transformation with $[L]_\mathscr{A} = A$.  Let $\zeta_i = \zeta^{q^i},
0\leq i\leq n - 1$ and let
\[
\begin{array}{l}
w_0 = v_0 + v_1 +\cdots + v_{n-1} \\
w_1 = \zeta_0 v_0 + \zeta_1 v_1 +\cdots + \zeta_{n-1}v_{n-1} \\
\vdots \\
w_{n-1} = \zeta^{n-1}_0 v_0 + \zeta^{n-1}_1 v_1 + \cdots +
\zeta^{n-1}_{n-1}v_{n-1} \; . \end{array}
\]
Then $Lw_i = w_{i+1}, 0\leq i\leq n-2$ and
\[ Lw_{n-1} = \zeta^n_0v_0 + \zeta^n_1v_1 + \cdots +
\zeta^n_{n-1}v_{n-1}\; . \]
But $a_0 + a_1\zeta_i + a_2\zeta^2_i + \cdots + \zeta^n_i = 0$ for all
$i$ and so
\[Lw_{n-1} = -a_0w_0 - a_1w_1\dots -a_{n-1}w_{n-1}\; .\]
Let $\mathscr{B} = \{w_1,w_1,\dots , w_{n-1}\}$.  Then
\[ [L]_\mathscr{B} =
\left( \begin{array}{ccccc}
0 & 0 & \cdots & 0 & -a_0 \\
1 & 0 & \cdots & 0 & -a_1 \\
0 & 1 & & & \\
\vdots & \vdots & \cdots & \vdots & \vdots \\
0 & 0 & \cdots & 1 & -a_{n-1} \; .\end{array}\right) \]
the companion matrix of $f(x)$.  Let
\[ V = \left( \begin{array}{llll}
1 & \zeta_0 & \cdots & \zeta^{n-1}_0 \\
1 & \zeta_1 & \cdots & \zeta^{n-1}_1 \\
\vdots & \vdots & & \vdots \\
1 & \zeta_{n-1} & \cdots & \zeta^{n-1}_{n-1} \end{array}\right) \]
be the Vandermonde matrix.  We have that
\[ V^{-1}AV = [L]_\mathscr{B}\in M_{n}(\mathbb{F}_q) \]
and by Lemma 2.3, since the fixed field of $\langle \tau\rangle$ is
$\mathbb{F}_q$,
\[ V^{-1}CV\in M_{n}(\mathbb{F}_q) \; . \]
Hence $\rho^G$ is realizable over $\mathbb{F}_q$.
\end{proof}

We are now ready to prove our main result.

\begin{theorem}
Let $G = \langle a,b\mid a^m = 1 = b^n, b^{-1}ab = a^k\rangle$ where
$|k|_m = n$.  As above, let $\rho$ be the representation of
$\langle a\rangle$ defined by $\rho (a) = \zeta$ where $\zeta$ is a
primitive $m^{th}$ root of $1$ in $\overline{\mathbb{F}_s}$, $(s,m) =
1$.  Then the induced representation $\rho^G$ is realizable over
$\mathbb{F}_q$ iff $q\equiv k^i(m)$ for some $i, 0\leq i\leq n-1$.
Here, as above, $q$ is a power of $s$.
\end{theorem}
\begin{proof}
Assume $\rho^G$ is realizable over $\mathbb{F}_q$.  Then since
$\rho^G(a) = \left( \begin{array}{cccc}
\zeta & & & 0 \\
& \zeta^k \\
& & \ddots \\
0 & & & \zeta^{k^{n-1}} \end{array}\right)$,
it follows that \d f(x) = \prod^{n-1}_0 (x - \zeta^{k^i})$, the
characteristic polynomial of $\rho^G(a)$ is in $\mathbb{F}_q[x]$.
Hence the Frobenius automorphism $\tau$ defined above permutes the
elements of the set
\[ \mathscr{P} = \{\zeta ,\zeta^k,\dots ,\zeta^{k^{n-1}}\}.\]  By Lemma 2.2, $q\equiv k^i(m)$ for some $i$.

Conversely, assume $q\equiv k^j(m)$ for some $j$.  Again by Lemma 2.2,
$\tau$ permutes the elements of $\mathscr{P} = \{\zeta, \zeta^k,\dots
,\zeta^{k^{n-1}}\}$.  Since $\langle\tau\rangle$ acts regularly on
$\mathscr{P}$, the orbits of $\langle\tau\rangle$ are all of the same
size.  Let there be $u$ orbits of size $v$ each so that $uv = n$.  Now
$|q|_m = v$ and since $q\equiv k^j(m)$ we have $1\equiv q^v\equiv
k^{vj}(m)$.  Thus $uv\mid vj$ whence $u\mid j$ and so $j = u\alpha$,
say.  Since $|k^u|_m = v = |q|_m$, it follows that
$(\alpha ,v) = 1$.  Therefore $|b^{u\alpha}| = v$.  Set $c =
b^{u\alpha}$ and let $H = \langle a,c\rangle$.  Then $c^{-1}ac =
a^{k^{u\alpha}} = a^q$.  Now $\langle\tau\rangle$ acts transitively on
$\{\zeta,\zeta^q,\dots ,\zeta^{q^{v-1}}\}$ and so, by Lemma 2.4,
$\rho^H$ is realizable over $\mathbb{F}_q$.  Hence $(\rho^H)^G$ is
realizable over $\mathbb{F}_q$.  But, by transitivity of induction,
$(\rho^H)^G$ is equivalent to $\rho^G$ whence $\rho^G$ is realizable
over $\mathbb{F}_q$.
\end{proof}

We now proceed to give a specific example of what happens when
$|k|_m\not= n$ (see Theorem 2.1).  It turns out that this
example mirrors faithfully the situation which obtains in the case of
the class of generalized quaternion groups $Q_{4m}, m$ odd.
\begin{example}
Let $Q_{12} = \langle a,b\mid a^3 = 1 = b^4, b^{-1}ab = a^2\rangle$.  In this
case $|k|_3 = 2\not= |b|$.  The conjugacy classes of $Q_{12}$ are
$\{1\},\{b^2\},\{a,a^2\},\{b,ab,a^2b\}$, $\{b^3,ab^3,a^2b^3\}$,
$\{ab^2,a^2b^2\}$.  Hence over $\overline{\mathbb{F}_s}$ there are six
irreducible representations, of which there are clearly four of degree one and
two of degree two. Letting $i$ be one of the square roots of $-1$ we obtain
the following character table:
\[ \begin{array}{r|r|r|r|r|r|r|}
& 1 & a & b^2 & b & b^3 & ab^2 \\ \hline
\chi_1 & 1 & 1 & 1 & 1 & 1 & 1 \\ \hline
\chi_2 & 1 & 1 & -1 & i & -i & -1 \\ \hline
\chi_3 & 1 & 1 & -1 & -i & i & -1 \\ \hline
\chi_4 & 1 & 1 & 1 & -1 & -1 & 1 \\ \hline
\chi_5 & 2 & -1 & 2 & 0 & 0 & -1 \\ \hline
\chi_6 & 2 & -1 & -2 & 0 & 0 & 1 \\ \hline
\end{array} \]

$\chi_5$ is afforded by the representation
\[ \rho_5(a) = \left( \begin{array}{rr}
0 & -1 \\
1 & -1 \end{array} \right)\; ,\quad \rho_5(b) = \left( \begin{array}{ll} 1 & -1 \\
0 & -1 \end{array} \right) \]
while $\chi_6$ is afforded by
\[ \rho_6(a) = \left( \begin{array}{ll}
\zeta & 0 \\
0 & \zeta^2 \end{array}\; \right)\; ,\quad \rho_6(b) = \left( \begin{array}{lr}
0 & -1 \\
1 & 0 \end{array} \right) \; . \]

The representation $\rho^G$ is induced from $\rho :\langle
a\rangle\to\langle\zeta\rangle$ where $\zeta$ is a primitive third root
of $1$ is
\[ \rho^G(a) = \left( \begin{array}{cccc}
\zeta & & & 0 \\
& \zeta^2 \\
& & \zeta \\
0 & & & \zeta^2 \end{array}\right) \]
\[ \rho^G(b) = \left( \begin{array}{cccc}
0 & 0 & 0 & 1 \\
1 & 0 & 0 & 0 \\
0 & 1 & 0 & 0 \\
0 & 0 & 1 & 0 \end{array}\right) \]
and this is easily seen to be equivalent to $\rho_5 \dot{+} \rho_6$:
\[ a\to \left( \begin{array}{rrrl}
0 & -1 & 0 & 0 \\
1 & 0 & 0 & 0 \\
0 & 0 & \zeta & 0 \\
0 & 0 & 0 & \zeta^2 \end{array}\right) \]
\[ b\to \left( \begin{array}{rrrr}
0 & 1 & 0 & 0 \\
1 & 0 & 0 & 0 \\
0 & 0 & 0 & -1 \\
0 & 0 & 1 & 0 \end{array}\right) \; . \]
Thus to find an explicit matrix representation of $\rho^G$ over some
$\mathbb{F}_q$, we need only realize $\rho_6$ over $\mathbb{F}_q$.  We do
this in full generality for $Q_{4m}$, $m$ odd.
\end{example}
\begin{prop}
Let $Q_{4m} = \langle a,b\mid a^m = 1 = b^4,b^{-1}ab = a^{-1}\rangle$
and let $\rho :\langle a\rangle\to \overline{\mathbb{F}_s}$ be the
representation defined by $\rho(a) = \zeta$ where $\zeta$ is a primitive
$m^{th}$ root of unity.  Then $\rho^G$ is the direct sum of the
representations $\sigma$ and $\tau$ where
\[ \sigma(a) = \left( \begin{array}{ll}
\zeta & 0 \\
0 & \zeta^{-1} \end{array} \right)\; ,\quad \sigma(b) = \left( \begin{array}{cc}
0 & 1 \\
1 & 0 \end{array} \right) \]
and
\[ \tau(a) = \left( \begin{array}{cc}
\zeta & 0 \\
0 & \zeta^{-1} \end{array} \right)\; ,\quad \tau(b) = \left( \begin{array}{rr}
0 & -1 \\
1 & 0 \end{array} \right) \; . \]
\end{prop}
\begin{proof}
$\sigma$ and $\tau$ are both irreducible since $\sigma (a),\sigma (b)$
(respectively, $\tau (a),\tau (b)$) have no common eigenvector.  Also,
by Frobenius reciprocity,
\begin{eqnarray*}
(\rho^G,\sigma ) & = & (\rho ,\sigma_{\langle a\rangle}) = \frac{1}{m}[2
+ \zeta (\zeta +\zeta^{-1}) + \zeta^2(\zeta^2 +\zeta^{-2}) +\cdots
+\zeta^{m-1}(\zeta^{m-1} + \zeta^{-(m-1)})] \\
& = & 1 \; \mbox{since} \; \langle\zeta^2\rangle = \langle\zeta\rangle
\; \mbox{because $m$ is odd.}
\end{eqnarray*}
Similarly, $(\rho^G,\tau) = 1$.  Thus, since $deg \; \sigma = deg \; \tau = 2$
and $deg \; \rho^G = 4$, it follows that $\rho^G = \sigma \dot{+} \tau$.
\end{proof}
{\bf Remark.}
The representation $\sigma$ is a representation of $Q_{4m}/\langle
b^2\rangle \approx\langle a,c\mid a^m = 1 = c^2, c^{-1}ac =
a^{-1}\rangle$ and can be dealt with using Theorem 2.5.  (See also
Theorem 3.6.)  We are thus left with the computation of an explicit form
for $\tau$ over $\mathbb{F}_q$, where, as usual, $q$ is a power of the
prime $s$.  Naturally, if $\tau$ is realizable over $\mathbb{F}_q$, then
$\zeta + \zeta^{-1}\in \mathbb{F}_q$.  Moreover, $\zeta +
\zeta^{-1}\in\mathbb{F}_q$ iff $q\equiv\pm 1(m)$ (see Theorem 3.6).
{\it Thus, given that} $\theta = \zeta + \zeta^{-1}\in\mathbb{F}_q$,
{\it we construct an explicit representation over} $\mathbb{F}_q$
{\it which is equivalent to} $\tau$.

We may assume $\zeta\notin\mathbb{F}_q$, otherwise $\tau$ itself is an
$\mathbb{F}_q$-representation.  Clearly the irreducible polynomial of
$\zeta$ over $\mathbb{F}_q$ is $f = x^2 - \theta x + 1$ and so $\zeta =
\frac{\theta + \sqrt{\theta^2-4}}{2}$ where we have taken a specific
square root of $\theta^2 - 4$ in $\overline{\mathbb{F}_q}$.

Choose $\alpha ,\beta\in\mathbb{F}_q$ such that $\alpha^2 + \beta^2 =
\theta^2 - 4$.  Since $\zeta\notin \mathbb{F}_q$ it follows that
$\beta\not= 0$.  Let $A = \left( \begin{array}{rr}
\alpha & \beta \\
\beta & -\alpha \end{array} \right)$ and set $\tilde{A} =
\frac{\theta}{2} + \frac{1}{2}A$.  By our choice of $\alpha$ and
$\beta$, we have $A^2 = (\theta^2 - 4)I$ and so the minimum polynomial
of $\tilde{A}$ is $f = x^2 - \theta x + 1$ since $\tilde{A}$ is not
diagonal.  The mapping $\varphi :\mathbb{F}_q[\zeta ]\to
M_2(\mathbb{F}_q)$ defined by
\[ \varphi(g(\zeta )) = g(\tilde{A}) \]
is well defined since the irreducible polynomial of $\zeta$ over
$\mathbb{F}_q$ is the same as the minimum polynomial of $\tilde{A}$, and
so embeds $\mathbb{F}_q [\zeta]$ in $M_2(\mathbb{F}_q)$.

Let $B = \left( \begin{array}{rr}
0 & -1 \\
1 & 0 \end{array} \right)$.  Then $B^{-1}AB = -A$ and so
$B^{-1}\tilde{A}B = \frac{\theta}{2}I - \frac{1}{2}A = \tilde{A}^{-1}$.
The map $\mu :Q_{4m}\to M_2(\mathbb{F}_q)$ defined by
\[ \mu(a) = \tilde{A}, \; \mu(b) = B \]
is clearly a faithful irreducible representation of $Q_{4m}$.   We show
that  $tr \tilde{A}^t = \zeta^t + \zeta^{-t}$.  Define $\gamma_1 = \theta$,
$\gamma_0 = 1$ and for $j\geq 2$ let $\gamma_j = \theta \gamma_{j-1}
-\gamma_{j-2}$. We claim that $\zeta^t = \gamma_{t-1}\zeta - \gamma_{t-2}$ for
all $t\geq 2$. Indeed, when $t = 2, \zeta^2 = \theta\zeta - 1 = \gamma_1\zeta
- \gamma_0$. Assume the result true for $t$.  Then
\begin{eqnarray*}
\zeta^{t+1} & = & \gamma_{t-1}\zeta^2 - \gamma_{t-2}\zeta \\
& = & \gamma_{t-1}(\theta\zeta - 1) - \gamma_{t-2}\zeta \\
& = & (\theta\gamma_{t-1} - \gamma_{t-2})\zeta - \gamma_{t-1} \\
& = & \gamma_t\zeta - \gamma_{t-1} \; . \quad \mbox{ By induction, the result
follows. }
\end{eqnarray*}
In a similar fashion we have $\zeta^{-t} = \gamma_{t-1}\zeta^{-1} -
\gamma_{t-2}$.  Now $\tilde{A}^t = \gamma_{t-1}\tilde{A} -
\gamma_{t-2}I$ and so $tr(\tilde{A}^t) = \theta\gamma_{t-1} -
2\gamma_{t-1}$ since $tr(\tilde{A}) = \theta$.  But $\zeta^t +
\zeta^{-t} = \theta\gamma_{t-1} - 2\gamma_{t-1} = tr \tilde{A}^t$.

Furthermore, it is easily checked that $tr(\tilde{A}^tB^j) = 0$ if $j$
is odd and $tr(\tilde{A}^tB^2) = -\zeta^t - \zeta^{-t}$.  Hence the
character afforded by $\mu$ is identical to that afforded by $\tau$.
Since each is irreducible, it follows that $\mu$ and $\tau$ are
equivalent.

Returning again to the case $Q_{12} = \langle a,b\mid a^3 = 1 = b^4$,
$b^{-1}ab = a^{-1}\rangle$, let $\mathbb{F}_q = \mathbb{F}_5$.  Since
$5\equiv -1(3)$, it follows that $\zeta + \zeta^{-1}\in\mathbb{F}_5$
where $\zeta$ is a primitive 3$^{rd}$ root of unity.  Indeed, since $x^2
+ x + 1$ is the irreducible polynomial of $\zeta$ over $\mathbb{F}_5$,
we have $\zeta + \zeta^{-1} = -1 = \theta$.  Hence $\theta^2 - 4 = -3$.
Choose $\alpha^2 + \beta^2 = -3$, say $\alpha = 1 = \beta$.  Then
$A = \left( \begin{array}{rr}
1 & 1 \\
1 & -1 \end{array} \right)$  and
\begin{eqnarray*} \tilde{A} & = & \frac{-I}{2} + \frac{1}{2} \left( \begin{array}{rr}
1 & 1 \\
1 & -1 \end{array} \right) \\
& = & 2I + 3 \left( \begin{array}{rr}
1 & 1 \\
1 & -1 \end{array} \right) \\
& = & \left( \begin{array}{rr}
0 & 3 \\
3 & -1 \end{array} \right)\; . \end{eqnarray*}
It now follows from the general theory above that the representation
\[ a\to \left( \begin{array}{cc}
\zeta & 0 \\
0 & \zeta^{-1} \end{array} \right) \; ,\quad b\to \left(
\begin{array}{rr}
0 & -1 \\
1 & 0 \end{array} \right) \]
is equivalent to
\[ a\to \left( \begin{array}{rr}
0 & 3 \\
3 & -1 \end{array} \right) \; ,\quad b\to \left(
\begin{array}{rr}
0 & -1 \\
1 & 0 \end{array} \right) \]

\section{Applications related to the quadratic reciprocity law}

By choosing special groups in Theorem 2.5 we are able to obtain some
interesting number-theoretic relations.  In particular, we shall recover the
classical quadratic reciprocity laws as well as an interesting property
concerning ``values of cosine", (see [Syl].  This property was also
re-discovered and generalized in [M-R].)  It is clear that many other
interesting facts can be established by choosing appropriate groups.  This is
left to the interested reader, and authors who are certainly interested.
\begin{example}
{\bf The quadratic reciprocity law.}  Let $p$ and $s$ be distinct odd primes,
$g$ a primitive root modulo $p$ and $\zeta$ a primitive $p^{th}$ root of unity
in $\overline{\mathbb{F}_s}$.  Let $G = G(p) = \langle a,b\mid a^p = 1 =
b^{\frac{p-1}{2}}, b^{-1}ab = a^{g^2}\rangle$ (see \cite[Section 5.15]{jac}).
Since $|g^2|_p = |b|$, we are in a position to apply Theorem 2.5.  Thus let
$\rho = \rho(p) : \langle a\rangle\to \overline{\mathbb{F}_s}$ be defined by
$\rho(a) = \zeta$ and let $(s/p)$ be the usual Legendre symbol.  Then we have
the following three theorems.
\end{example}
\begin{theorem}
$\rho^G$ is realizable over $\mathbb{F}_s$ iff $(\frac{s}{p}) = 1$.
\end{theorem}
\begin{theorem}
$\rho^G$ is realizable over $\mathbb{F}_s$ iff $(\frac{p^*}{s}) = 1$
where $p^* = (-1)^{\frac{p-1}{2}}p$.
\end{theorem}
\begin{corollary} \label{Gauss}
(The quadratic reciprocity law [Gau])
\[ \left(\frac{s}{p} \right) = \left(\frac{p^*}{s} \right) \; . \]
\end{corollary}
\begin{proof}[Proof of Theorem 3.2.]
By Theorem 2.5, $\rho^G$ is realizable over $\mathbb{F}_s$ iff $s\equiv
g^{2i}(p)$ for some $i, 0\leq i\leq \frac{p-3}{2}$.  That is, $\rho^G$
is realizable over $\mathbb{F}_s$ iff $s$ is a square modulo $p$.
\end{proof}

\noindent
{\bf Remark.}
Observe that the realizability of $\rho^G$ over $\mathbb{F}_s$ depends
entirely on $s$ (mod $p$)!
\begin{proof}[Proof of Theorem 3.3.]
By Theorem~3.2, if $\rho^G$ is realizable over $\mathbb{F}_s$,
then the Frobenius automorphism $x\to x^s$ on $\overline{\mathbb{F}}_s$
permutes the elements of the set $\{\zeta,\zeta^{g{^2}},\dots ,\zeta^{g^{p-3}}\}$.
But then \d \sum^{(p-3)/2}_{i=0} \zeta^{g^{2i}} \in
\mathbb{F}_s$.  On the other hand (\cite[Section 5.15]{jac})  we have \d
\sum^{(p-3)/2}_{i=0} \zeta^{g^{2i}} = \frac{-1\pm\sqrt{p^*}}{2}$, whence
$\sqrt{p^*} \in  \mathbb{F}_s$ and so $(\frac{p^*}{s}) = 1$.

Conversely, assume that $(\frac{p^*}{s}) = 1$.  Then, as above, \d
\sum^{(p-3)/2}_{i=0} \zeta^{g^{2i}}\in\mathbb{F}_s$.  We show that the
Frobenius automorphism $\varphi:x\to x^s$ on $\overline{\mathbb{F}_s}$
permutes the elements of the set $\{\zeta,\zeta^{g^2},\zeta^{g^4},\dots
,\zeta^{g^{p-3}}\}$. Then using Theorem 2.5 we conclude that $\rho^G$
is realizable over $\mathbb{F}_s$.

To do this, it is sufficient by Lemma 2.2, to show that
$\varphi(\zeta) = \zeta^{g^{2j}}$ for some $j\in\{0,1,\dots ,\frac{p-3}{2}\}$.
Now \d 1 + \sum^{(p-3)/2}_{i=0} \zeta^{g^{2i}} + \sum^{(p-3)/2}_{i=0}
\zeta^{g^{2i+1}} = 0$.  If $\phi(\zeta) = \zeta^{g^{2j+1}}$, then \d
\sum^{(p-3)/2}_{i=0} \zeta^{g^{2i}} = \sum^{(p-3)/2}_{i=0} \zeta^{g^{2i+1}}$
whence \d 1 + 2 \sum^{(p-3)/2}_{i=0} \zeta^{g^{2i}} = \pm\sqrt{p^*} = 0$. Thus
$p^* = 0$, contradicting the hypothesis that $p$ and $s$ are distinct primes.
Therefore, $\rho^G$ is realizable over $\mathbb{F}_s$ and Corollary \ref{Gauss} follows
immediately.
\end{proof}

\noindent
{\bf Remarks.}
(1)  The method of evaluation \d \sum_{i=0}^{\frac{p-3}{2}}
\zeta^{g^{2i}}$ given in \cite[Section 5.15]{jac} can be simplified and
clarified as follows:

Let $Q$ be the set of non-zero quadratic residues mod $p$ and let $N$ be the set of
nonquadratic residues mod $p$.  If $\zeta$ is a primitive $p^{th}$ root of
unity in $\overline{\mathbb{F}_s}$ we have:
\[ (i) \; \sum_{x\in Q} \zeta^{-x} = \sum_{x\in Q} \zeta^x \quad \mbox{and }
\sum_{x\in N} \zeta^x = \sum_{x\in N} \zeta^{-x} \quad \mbox{if } p\equiv
1(4)\; ; \]
\[ (ii) \; \sum_{x\in Q} \zeta^{-x} = \sum_{x\in N} \zeta^x \quad \mbox{and }
\sum_{x\in N} \zeta^{-x} = \sum_{x\in Q} \zeta^{x} \quad \mbox{if } p\equiv
-1(4)\; . \]

Now let \d c = \sum_{x\in Q} \zeta^x$ and let $G = \langle a,b\mid a^p =
1 = b^{\frac{p-1}{2}}, b^{-1}ab = a^{g^2}\rangle$ where $g$ is a
primitive root mod $p$.  Let $\rho_1$ and $\rho_2$ be the
representations of $H = \langle a\rangle$ defined by $\rho_1(a) = \zeta$
and $\rho_2(a) = \zeta^g$ and let $\chi_1$ and $\chi_2$ denote their
respective characters.  Since $|g^2|_p = \frac{p-1}{2} = |b|$,
it follows by Proposition 2.1 (1) that $\rho^G_1$ and $\rho^G_2$
are irreducible.  Moreover, since $\rho^G_1(a)$ and $\rho^G_2(a)$ have
different eigenvalues, they are inequivalent.  Hence by Schur's Lemma
and Frobenius reciprocity we have

\[0 = (\chi^G_2 ,\chi^G_1) = (\chi_2 ,(\chi^G_1)_H)\; .\]
But
\[(\chi_2 ,(\chi^G_1)_H) = \frac{1}{|H|}\left(\frac{p-1}{2} +
\zeta^g\sum_{x\in Q} \zeta^{-x} + \zeta^{g^2}\sum_{x\in N} \zeta^{-x} +
\zeta^{g^3}\sum_{x\in Q} \zeta^{-x} +\cdots + \zeta^{g^{p-1}}\sum_{x\in
N}\zeta^{-x}\right)\; . \]

{\bf Case (i). } $p\equiv 1(4)$.  By (i) above we have
\[ 0 = \frac{p-1}{2} + \zeta^g c + \zeta^{g^2}(-1 - c) + \zeta^{g^3} c
+\cdots + \zeta^{g^{p-1}}(-1 - c). \]
Hence
\[ \frac{1-p}{2} = c(-1 - c) + c(-1 - c) \]
and so
\[ 4c^2 + 4c + 1 - p = 0 \; . \]
Therefore
\[ c = \frac{-1\pm\sqrt{p}}{2}\; . \]

{\bf Case (ii). } $p\equiv -1(4)$.  By (ii) above we have
\[ 0 = \frac{p-1}{2} + \zeta^g(-1 - c) + \zeta^{g^2}c +\cdots +
\zeta^{g^{p-1}}c\; . \]
Hence
\[ \frac{p-1}{2} + (-1 - c)^2 + c^2 = 0 \]
yielding
\[ 4c^2 + 4c + p + 1 = 0\; . \]
Therefore
\[ c = \frac{-1\pm\sqrt{-p}}{2}\; . \]

We may combine the two solutions by setting $p^* = (-1)^{p-1/2}p$.  Then
$c = \frac{-1\pm\sqrt{p^*}}{2}$.

(2)  Assume that $p$ is an odd prime such that $p\equiv 1(n)$.  Let $g$
be a primitive root modulo $p$ and let $G(p,n) = \langle a,b\mid a^p = 1
= b^{(p-1)/n}, b^{-1}ab = a^{g^n}\rangle$.  The group $G(p)$ of the
previous example is the group $G(p,2)$ in this notation.  Observe that
$|g^n|_p = \frac{p-1}{n} = |b|$ and so, once again,
Theorem 2.5 is applicable.  We find immediately that the representation
$\rho^{G(p,n)}$ is realizable over $\mathbb{F}_s$ iff $s$ has an $n^{th}$
root in $\mathbb{F}_p$.  In this case again we see a connection with
higher reciprocity laws (see, for example [I-R], [Ank]) which we plan to
investigate in a subsequent paper.

\begin{example}
In his paper [Syl], Sylvester discovered that if $\zeta$ is a primitive
m$^{th}$ root of unity in $\overline{\mathbb{F}_s}$, $(s,m) = 1$, then
$2 cos\frac{2\pi}{m}\; : = \zeta + \zeta^{-1}$ belongs to
$\mathbb{F}_s$ iff $s\equiv\pm 1(m)$.  Sylvester's formulation of his
result differs somewhat from that just stated, but is equivalent to it.
In the paper [M-R] this result was re-discovered, generalized and
applied to some questions concerning extensions of degree $2^{l}$.

Here we shall strengthen Sylvester's result by applying our Theorem 2.5
to the dihedral group.  As usual, set $D_{2m} = \langle a,b\mid a^m = 1 =
b^2, b^{-1}ab = a^{-1}\rangle, m\in\mathbb{N}, m\geq 3$.  Then
$|-1|_m = 2 = |b|$ and so our theorem is applicable.  We have
immediately
\end{example}
\begin{theorem}
The $2$-dimensional representation $\rho$ of $D_{2m}$ given by
\[ \rho(a) = \left( \begin{array}{cc}
\zeta & 0 \\
0 & \zeta^{-1}\end{array}\right)\; , \quad \rho(b) =
\left( \begin{array}{cc}
0 & 1 \\
1 & 0 \end{array}\right) \]
where $\zeta$ is a primitive $m^{th}$ root of unity is realizable over
the field $\mathbb{F}_q$, $q$ a power of the odd prime $s$, iff
$q\equiv\pm 1(m)$.  Moreover, in the case when $\rho$ is realizable over
$\mathbb{F}_q$ we can write the corresponding matrices with entries in
$\mathbb{F}_q$ explicitly.  Indeed:

(1)  if $q\equiv 1(m)$, then $\zeta\in\mathbb{F}_q$ and the original
matrices lie in $\mathbb{F}_q$;

(2)  if $q\equiv -1(m)$, then $\zeta\notin\mathbb{F}_q$ but $t = \zeta +
\zeta^{-1}\in\mathbb{F}_q$ and $\rho$ is equivalent to the representation
$\theta$ given by
\[ \theta (a) = \left(\begin{array}{cc}
0 & 1 \\
1 & t \end{array}\right)\; , \quad \theta (b) =
\left(\begin{array}{rr}
1 & t \\
0 & -1 \end{array}\right)\; . \]
\end{theorem}

One can find the matrices over $\mathbb{F}_q$ following the proof of Theorem
2.5.  From the explicit form of the matrices we deduce immediately that $\zeta
+ \zeta^{-1}\in\mathbb{F}_q$ iff $q\equiv\pm 1(m)$ which is indeed a
generalization of Sylvester's result.  This theorem does give us additional
information, namely, that if $q\equiv\pm 1 (m)$, then, not only is $\zeta +
\zeta^{-1}\in\mathbb{F}_q$, but the whole representation $\rho$ is realizable
over $\mathbb{F}_q$ and moreover, explicit formulas for the matrices $\rho(g),
g\in D_{2m}$ can be computed.

\section{Connections with cross-products}

The methods used in the previous sections are of an elementary nature
but may appear somewhat mysterious to the reader.  The veil of mystery
lifts however and we gain considerable insight into our computations
once we establish a connection with crossed products (see for example
\cite{Her} or \cite{rei}).  Moreover, guided by this connection with crossed products
we are able to obtain a stronger result concerning complete
realizability (cf. Definition 4.5).

Roughly speaking, cross-products intervened in the following manner.
We consider a finite field $F=\mathbb{F}_q$ over which we want to
construct a representation module $M$
for our group $G$ which realizes a given component $\rho_i,i=0,1,\dots,
r-1$ of $\rho^G$. (See Proposition~2.1.)  More precisely, we want to
verify that for $L=F(\zeta)$, where $\zeta$ is a primitive $m^{th}$ root of
unity, the action of $a,b \in G$ on $L \otimes_F M$ has the required
form described in Proposition~2.1(2) with respect to a suitable basis
of $L \otimes_F M$.

Now a hint on how to construct the required representation module
$M$ is obtained via the cross-product $A=(L/F,\mathcal{G},f)$ with
a trivial factor set $f$ (see \cite[Chapter~4]{Her} and Example~4.1
below) in the case when $[L:F]=t,q\equiv k$ (mod $m$), $q\equiv 1$
(mod $r$), and $(r,m)=1$. (In the notation of Proposition~2.1.)

Then the idea is to choose $M\cong L$ as $F$-vector spaces and use
two facts.
\begin{enumerate}
  \item We can embed our group $G$ into $A$.
  \item There exists an isomorphism $\varphi\colon A \to \Hom_F
  (L,L)$.
\end{enumerate}

Then using $\varphi$ restricted to $G$ embedded in $A$, we obtain
a representation $G$ on $M$.

Thus we can say that our cross-product $A$ guides us to make the
specific representation of $G$ on $M$ described in the first
paragraph of our proof of Theorem~4.5.

We should point out, however, that we only use the cross-product
construction as a guide for building a representation
$M$, and our further exposition is logically independent of this
construction.  Nevertheless it seems to us worthwhile to include
at least this idea, and to explain it in a detailed way in Example~4.1
below.  One could say that if we were to follow C. F. Gauss's style
of exposition, we would dismantle the scaffolding upon completion
of the building. We have instead tried to follow L.~Euler, by
leaving the scaffolding intact.

We begin with an example which points the way in the general case.
\begin{example}
Let $G = \langle a,b\mid a^7 = 1 = b^9, b^{-1}ab = a^2\rangle$.  Let $F$
be a finite field of order $q$ where $q\equiv 2(7)$ and $q\equiv 1(3)$
(for example, $q = 37$ would do).  Let $\zeta$ be a primitive 7$^{th}$
root of unity in $\overline{F}$.  In this case $|2|_7 = 3$ and
so, using the notation established in Proposition 2.1, $r = t = 3$ and
$\rho^G$ is the direct sum of the representations
\[ a\to \left(\begin{array}{lll}
\zeta & 0 & 0 \\
0 & \zeta^2 & 0 \\
0 & 0 & \zeta^4 \end{array}\right) \quad\quad
b\to \left(\begin{array}{ccc}
0 & 0 & \eta \\
1 & 0 & 0 \\
0 & 1 & 0 \end{array}\right) \]
as $\eta$ varies over the third roots of unity in $F$.  Let $L =
F(\zeta)$.  Then $\mathscr{G} = Gal(L/F) = \langle\sigma\rangle$ where
$\sigma(\zeta) = \zeta^2$.  Form the crossed product $A =
(L/F,\mathscr{G},f)$ with trivial factor set $f$ (cf. \cite[Chapter 4]{Her}).  Recall
that $A$ is a 3-dimensional algebra over $L$ with basis $1 =
u_1,u_{\sigma},u_{{\sigma}^2}$ and multiplication defined according to
the following:

(i)  $u_{\tau}u_{\nu} = u_{\tau\nu}\; ; \quad\quad\quad$ (ii)  $u_{\tau}\ell =
\tau(\ell)u_{\tau}\;$.

It is easily checked (cf. \cite[Chapter 7, Section 29]{rei}) that the map
$\varphi : A\to \Hom_F(L,L)$ defined by
\[ [\varphi (\ell_0 + \ell_1u_{\sigma} +
\ell_2u_{{\sigma}^2})](\lambda) = \ell_0\lambda + \ell_1\sigma(\lambda) +
\ell_2\sigma^2(\lambda) \]
for all $\lambda\in L$ is an $F$-algebra
isomorphism.

Now let $\eta$ be a third root of unity in $F$ and choose $z\in L$ such that
$N_{L/F}(z) = \eta$.  Here $N_{L/F}$ is the norm map from $L$ down to $F$. Let
$\psi : G\to A$ be the homomorphism defined by
\[ \psi(a) = \zeta,\quad \psi(b) = zu_{{\sigma}^2}\; . \]
Since $\zeta^7 = 1$, $(zu_{{\sigma}^2})^3 = \eta$ and
$(zu_{{\sigma}^2})\zeta (zu_{{\sigma}^2})^{-1} = \zeta^4$ it follows that
$\psi$ is indeed a well-defined homomorphism. (From the last identity
it follows that $(zu_{\sigma^2})^{-1}\zeta(zu_{\sigma^2})=\zeta^2$.)

Let $f = (x-\zeta)(x-\zeta^2)(x-\zeta^4)\in F[x]$.  Then $f$ is irreducible
polynomial since $\mathscr{G}$ acts transitively on $\{\zeta, \zeta^2, \zeta^4
\}$. Hence $L$ is isomorphic with $M:= \frac{F[x]}{(f)}$, where we send
$\zeta$ to $\bar{x}$, the class of $x$ mod $f$.
  We now define an action of $G$ on $M$ via
$\psi$ by
\end{example}
\begin{eqnarray*}
ag(\overline{x}) & = & \overline{x}g(\overline{x}) \\
bg(\overline{x}) & = & z(\overline{x})g(\overline{x}^4)
\end{eqnarray*}
where $z(\overline{x})$ corresponds to $z$.  This clearly turns $M$ into
an $FG$-module since $\psi$ is a homomorphism.

Let $M^L = L\otimes_F\frac{F[x]}{(f)} = \frac{L[x]}{(f)}$.  We show that
relative to an appropriate basis, $M^L$ affords the matrix
representation
$a\to \left(\begin{array}{lll}
\zeta & 0 & 0 \\
0 & \zeta^2 & 0 \\
0 & 0 & \zeta^4 \end{array}\right)\; ,\quad b\to
\left(\begin{array}{ccc}
0 & 0 & \eta \\
1 & 0 & 0 \\
0 & 1 & 0 \end{array}\right)\; .$
Indeed, let
\[ g_0 = \frac{(\overline{x} - \zeta^2)(\overline{x} - \zeta^4)}{(\zeta
- \zeta^2)(\zeta - \zeta^4)}\; ,\quad g_1 = \frac{z(\overline{x})(\overline{x}
- \zeta)(\overline{x} - \zeta^4)}{(\zeta^2 - \zeta)(\zeta^2 - \zeta^4)}\;
,\quad g_2 = \frac{z(\overline{x})z(\overline{x}^4)(\overline{x} -
\zeta)(\overline{x} - \zeta^2)}{(\zeta^4 - \zeta)(\zeta^4 - \zeta^2)}\; . \]
Using the equality $\bar{f} = 0$ in $M$, we see that
\[ ag_0 = \zeta g_0\; ,\quad ag_1 = \zeta^2g_1\quad\mbox{and}\quad ag_2
= \zeta^4g_2\; . \]
Moreover,
\[ bg_0(\overline{x}) = z(\overline{x})g_0(\overline{x}^4) =
\frac{z(\overline{x})(\overline{x}^4 - \zeta^2)(\overline{x}^4 -
\zeta^4)}{(\zeta - \zeta^2)(\zeta - \zeta^4)}\; . \]

But $\frac{(x^4 - \zeta^2)(x^4 - \zeta^4)}{(\zeta - \zeta^2)(\zeta -
\zeta^4)}\;$ and $\;\frac{(x - \zeta)(x - \zeta^4)}{(\zeta^2 - \zeta)(\zeta^2
- \zeta^4)}$ evaluated at $\zeta$, $\zeta^2$ and $\zeta^4$ both yield 0, 1, 0
respectively, thus showing that $bg_0 = g_1$.  Similarly $bg_1 = g_2$ and
$bg_2 = \eta g_0$ since $z(\overline{x})z(\overline{x}^4)z(\overline{x}^2) =
\eta$.  It follows that the given representation is realizable over $F$
because  $M$ is our desired $3$-dimensional representation space over $F$.

The above example is representative of the proof we are about to embark
upon for the general case, except that, in the general case we need a
technical maneuver to deal with the possibility that $f$ is reducible
over $F$ and that conjugation of $a$ by $b$ does not reflect the action
of the Galois group on $\zeta$.  In the example above $f$ is irreducible
over $F$ and $b^{-1}ab = a^2$ reflects the fact that $\sigma(\zeta) =
\zeta^2$.

We begin with a simple lemma.

\begin{lemma}
Let $F$ be a field and $f$ a monic polynomial over $F$.  Assume $f$ has
distinct roots $\zeta_1,\zeta_2,\dots ,\zeta_t$ in some splitting field
and let $d$ be a positive integer such that $\{\zeta_1,\zeta_2,\dots
,\zeta_t\} = \{\zeta^d_1,\zeta^d_2,\dots ,\zeta^d_t\}$.  Then $f(x)\mid
f(x^d)$.
\end{lemma}

\begin{proof}
\d f(x^d) = \prod_i (x^d - \zeta_i) = \prod_i (x^d - \zeta^d_i)$.  Clearly
each $\zeta_i$ is a root of $f(x^d)$ and so $f(x) \mid f(x^d)$.
\end{proof}

\begin{corollary}
With the same hypotheses and notation as above, if $g(x)$ and $h(x)$ are
polynomials over $F$ and $g(x)\equiv h(x)(f)$, then $g(x^d)\equiv
h(x^d)(f)$.
\end{corollary}

\begin{proof}
$f(x)\mid (g(x) - h(x))\Longrightarrow f(x^d)\mid (g(x^d) - h(x^d))$.  By the
above lemma we have $f(x)\mid (g(x^d) - h(x^d))$.
\end{proof}

The next lemma is the technical maneuver referred to above.

\begin{lemma}
Let $s$ be a prime, $m$ a positive integer with $(m,s) = 1$ and let $q$ be a
power of $s$.  Let $k$ be a positive integer and let $\zeta$ be a primitive
$m^{th}$ root of unity in $\overline{\mathbb{F}_s}$.  Assume that $|k|_m = t$
and that $q\equiv k^j(m)$ for some $j$.  Let $f = (x - \zeta)(x -
\zeta^k)\cdots (x - \zeta^{k^{t-1}})$ and let $\eta\in F = \mathbb{F}_q$.
Then there exists $z(x)\in F[x]$ such that $z(x)z(x^k)\cdots
z(x^{k^{t-1}})\equiv \eta(f)$.
\end{lemma}

\begin{proof}
The set $\{\zeta,\zeta^k,\dots ,\zeta^{k^{t-1}}\}$ of roots of $f$ in $L =
F(\zeta)$ is invariant under the Frobenius automorphism $a\to a^q$ since
$q\equiv k^j(m)$.  Hence $f\in F[x]$.  Let $f = f_1f_2\dots f_u$ be the
factorization of $f$ into irreducible factors over $F$ and let $S$ (resp.
$\widehat{S}$) denote the set of roots of $f_1$ (resp. $f_2f_3\dots f_u$).
Assume that $\zeta \in S$. Let $\widehat{f}_i = f/f_i, i = 1,2, \cdots , u$.
Since $(\widehat{f}_i,f_i) = 1$, there exists $h_i\in F[x]$ such that
$h_i\widehat{f}_i\equiv 1(f_i)$. Now $\frac{F[x]}{(f_1)}$ is isomorphic with
$L$ where $\bar{x}$ (the element $x$ mod $f_1$) plays the role of $\zeta$.
Further there exists $z\in L$ such that $N_{L/F}(z) = \eta$. Hence there
exists $z_1(x)\in F[x]$ such that \d \prod_{\xi\in S} z_1(\xi ) = \eta$. Now
let
\[ z(x) = h_1(x)\widehat{f}_1(x)z_1(x) + h_2(x)\widehat{f}_2(x) +\cdots +
h_u(x)\widehat{f}_u(x). \]
We observe that if $\alpha$ is a root of
$f_i$, $h_i(\alpha)\widehat{f}_i(\alpha) = 1$, while if $\alpha$ is a
root of $\widehat{f}_i$, $h_i(\alpha)\widehat{f}_i(\alpha) = 0$.  Now
let $\delta$ be a root of $f$.  We compute

\[ z(\delta)z(\delta^k)\dots z(\delta^{k^{t-1}}) = \prod_{\xi\in S}
z(\xi) \prod_{\xi\in\widehat{S}} z(\xi)\; . \]
If $\xi\in \widehat{S}$, (say $\xi$ is a root of $f_j$, $j\not= 1$),
then $z(\xi) = h_j(\xi)\widehat{f}_j(\xi) = 1$.  Thus
\[ z(\delta)z(\delta^k)\dots z(\delta^{k^{t-1}}) = \prod_{\xi\in S}
z(\xi)\; . \]
But if $\xi\in S$, $z(\xi) = z_1(\xi)$ and so
\[ z(\delta)z(\delta^k)\dots z(\delta^{k^{t-1}}) = \prod_{\xi\in S}
z_1(\xi) = \eta\; . \]
Hence $z(x)z(x^k)\dots z(x^{k^{t-1}})$ evaluated at any root of $f$
yields $\eta$.  It follows that
\[z(x)z(x^k)\dots z(x^{k^{t-1}})\equiv\eta(f).\]
\end{proof}
\begin{theorem}
Let $G = \langle a,b\mid a^m = b^n = 1, b^{-1}ab = a^k\rangle$ and let
$|k|_m = t$, $n = rt$.  Let $\zeta$ be a primitive $m^{th}$ root
of unity in $\overline{\mathbb{F}_s}$, $s$ a prime with $(s,m) = 1$.
Let $q$ be a power of $s$ and assume $q\equiv k^j(m)$ for some $j$.
Assume further that $F (= \mathbb{F}_q)$ contains a primitive $r^{th}$
root of unity $\eta$.  Then, for each integer $c$, the representation of
$G$ defined by
\[ a\to \left(\begin{array}{cccc}
\zeta & & & 0 \\
& \zeta^k \\
& & \ddots \\
0 &  & & \zeta^{k^{t-1}} \end{array}\right)\; ,\quad b \to
\left(\begin{array}{ccccc}
0 & 0 & \cdots & 0 & \eta^c \\
1 & 0 & \cdots & 0 & 0 \\
0 & 1 & \cdots & \vdots & \vdots \\
\vdots & \vdots & \cdots & 0 \\
0 & 0 & \cdots & 1 & 0 \end{array}\right) \]
is realizable over $F$.
\end{theorem}
\begin{proof}
Let $f = (x - \zeta)(x - \zeta^k)\dots (x - \zeta^{k^{t-1}})$.  As in
Lemma 4.4, $f\in F[x]$.  Let $M = \frac{F[x]}{(f)}$ and turn $M$ into
an $FG$-module by defining
\[ ag(\overline{x}) = \overline{x}g(\overline{x}) \]
and
\[ bg(\overline{x}) = z(\overline{x})g(\overline{x}^{\varphi(b^{-1})}) \]
where $z(x)$ is chosen as in Lemma 4.4 with respect to $\eta^c$ and $\varphi
:\langle b\rangle\to\langle [k]_m\rangle$ is the homomorphism defined by
$\varphi (b) = [k]_m$.  Recall that by $[k]_m$ we mean $k$ (mod $m$) and
naturally, by $\overline{x}^{[i]_m}$ we mean $\overline{x}^i$.  This is
independent of the representative of $[i]_m$ since $\overline{x}^m =
\overline{1}$.  Clearly the action of $a$ is well defined. That the action of
$b$ is well defined follows from Corollary 4.3. Straightforward computation
yields $b^ng(\overline{x}) = g(\overline{x})$ while it is obvious that
$a^mg(\overline{x}) = g(\overline{x})$.  In addition $abg(\overline{x}) =
\overline{x}z(\overline{x})g(\overline{x}^{\varphi (b^{-1})})$ while
\[ ba^kg(\overline{x}) = b\overline{x}^k g(\overline{x}) =
z(\overline{x})\overline{x}^{\varphi (b)\varphi(b^{-1})}g
(\overline{x}^{\varphi (b^{-1})}) = \overline{x}z(\overline{x})g
(\overline{x}^{\varphi (b^{-1})})\; . \] It now follows that we have a well
defined action of $G$ on $M$, thus turning $M$ into an $FG$-module.  Let $L =
F(\zeta)$ and let $M^L = L\otimes_F M = \frac{L[x]}{(f)}$. Then $M^L$ affords
the same matrix representation as $M$ relative to the basis
$\{1,\overline{x},\dots ,\overline{x}^{t-1}\}$.  We construct a basis
$\mathscr{B}$ of $\frac{L[x]}{(f)}$ such that relative to $\mathscr{B}$, the
matrix representation afforded by $M^L$ is the given representation.

Let $g_i(\overline{x}) = \frac{f(\overline{x})}{(\overline{x} - \zeta^{k^i})f'
(\zeta^{k^i})}\; ,\quad i = 0,1,\dots , t - 1$ and define
\[ h_0(\overline{x}) = g_0(\overline{x}),\quad h_1(\overline{x}) =
z(\overline{x})g_1(\overline{x}), \]
\[h_2(\overline{x}) =
z(\overline{x})z(\overline{x}^{\varphi(b^{-1})})g_2(\overline{x}),\dots
, h_{t-1}(\overline{x}) =
z(\overline{x})z(\overline{x}^{\varphi(b^{-1})})\cdots
z(\overline{x}^{\varphi(b^{-(t-2)})})g_{t-1}(\overline{x})\; . \]

First observe that $(\overline{x} - \zeta^{k^i})g_i(\overline{x}) =
\overline{0}$ and so $\overline{x}g_i(\overline{x}) =
\zeta^{k^i}g_i(\overline{x})$ whence $ah_i(\overline{x}) =
\zeta^{k^i}h_i(\overline{x})$ for $i = 0,1,\dots , t-1$.  Consider now
\[ bh_i(\overline{x}) =
z(\overline{x})h_i(\overline{x}^{\varphi(b^{-1})}) =
z(\overline{x})z(\overline{x}^{\varphi(b^{-1})})\cdots
z(\overline{x}^{\varphi(b^{-i})})g_i(\overline{x} ^{\varphi(b^{-1})})\; . \]
We show that $g_i(\overline{x}^{\varphi(b^{-1})}) = g_{i+1}(\overline{x})$ for
all $0\leq i\leq t - 1$ where $i$ is taken modulo $t$.  Indeed,
\[ g_i(\overline{x}^{\varphi(b{-1})}) = \prod_{v\not= i}
(\overline{x}^{k^{t-1}} - \zeta^{k^v})\left/\right. \prod_{v\not= i}
(\zeta^{k^i} - \zeta^{k^v})\; . \]
Clearly $g_i(x^{\varphi(b{-1})})$ vanishes for all $\zeta^{k^v}$ except
$\zeta^{k^{i+1}}$ when its value is $1$.  The same holds for
$g_{i+1}(x)$ and so we have $bh_i(\overline{x}) = h_{i+1}(\overline{x})$
provided $1\leq i\leq t - 2$.  Moreover, $bh_{t-1}(\overline{x}) =
z(\overline{x})z(\overline{x}^k)\cdots
z(\overline{x}^{k^{t-1}})g_{t-1}(\overline{x}^{\varphi(b^{-1})}) =
\eta^ch_0(\overline{x})$.  Therefore, as claimed, $M^L$ affords the same
matrix representation as the original one.
\end{proof}
\noindent
{\bf Remark: } We have an algorithm for obtaining the $F$-representation
from the given $F(\zeta)$-representation (once we have found $z(x)!$),
namely:  the matrix for $a$ is the companion matrix of $f$; the $(i +
1)^{st}$ column of the matrix for $b$ is computed as follows:  write
\[ z(x)x^{ik^{t-1}} = f(x)g(x) + r(x)\]
where $\deg \, r(x) <  \deg \,  f(x)$.  Then
the column vector formed by the coefficients of $r(x)$ (coefficient of
constant term first) is the $(i+1)^{st}$ column of the matrix for $b$.
\begin{definition}
Let $L$ be an extension field of $F$ and let $\rho$ be an
$L$-representation of a group $G$.  We say $\rho$ is {\it completely
realizable} over $F$ if $\rho$ is equivalent to an $F$-representation of
$G$, each of whose irreducible components over $L$ is realizable over
$F$.
\end{definition}
\begin{corollary}
Let the notation be as in Theorem 4.4 and let $\rho$ be the
representation $\langle a\rangle$ defined by $\rho(a) = \zeta$.  Then
$\rho^G$ is completely realizable over $F (= \mathbb{F}_q)$ if and only
if $q\equiv k^j(m)$ for some $j$ and $q\equiv 1(r)$.
\end{corollary}
\begin{proof}
By Proposition 2.1, $\rho^G$ is a sum of representations $\rho_i$ of the
form dealt with in Theorem 4.5.  The sufficiency is thus established.
For the necessity we observe first that if $\rho^G$ is completely
realizable over $F$, then $F$ must contain a primitive $r^{th}$ root of
unity since (using the notation of Proposition 2.1 (2)) the
characteristic polynomial of $\rho_1(b)$ is $x^t - \eta^{-1}$.  In
addition the characteristic polynomial of $\rho_1(a)$ remains invariant
under the Frobenius automorphism $\tau : y\to y^q$ and so $\tau$ must
permute the elements of the set $\{\zeta,\zeta^k,\dots ,\zeta^{k^{t-1}}\}$
and so $q\equiv k^j(m)$.
\end{proof}

We finish by applying the results obtained to compute a specific
example.
\begin{example}
Let $G = \langle a,b\mid a^5 = 1 = b^8$, $b^{-1}ab = a^2\rangle$, $s = q =
19$. In this case $|2|_5 = 4$ and so $t = 4, r = 2$.  Also $f = x^4 + x^3 +
x^2 + x + 1$.  By the general theory, $\rho^G$ is that direct sum of the two
irreducible representations $\rho_1$ and $\rho_2$ defined by
\[ \rho_1(a) = \left(\begin{array}{cccc}
\zeta & 0 & 0 & 0 \\
0 & \zeta^2 & 0 & 0 \\
0 & 0 & \zeta^4 & 0 \\
0 & 0 & 0 & \zeta^3 \end{array}\right)\; ,\quad \rho_1(b) =
\left(\begin{array}{cccc}
0 & 0 & 0 & 1 \\
1 & 0 & 0 & 0 \\
0 & 1 & 0 & 0 \\
0 & 0 & 1 & 0 \end{array}\right) \]
and
\[ \rho_2(a) = \left(\begin{array}{cccc}
\zeta & 0 & 0 & 0 \\
0 & \zeta^2 & 0 & 0 \\
0 & 0 & \zeta^4 & 0 \\
0 & 0 & 0 & \zeta^3 \end{array}\right)\; ,\quad \rho_2(b) =
\left(\begin{array}{cccr}
0 & 0 & 0 & -1 \\
1 & 0 & 0 & 0 \\
0 & 1 & 0 & 0 \\
0 & 0 & 1 & 0 \end{array}\right)\; . \]
We first compute a matrix representation $\widehat{\rho}_1$ over
$\mathbb{F}_{19}$ equivalent to $\rho_1$ using the algorithm established
above.  We know $\widehat{\rho}_1(a) = \left(\begin{array}{cccc}
0 & 0 & 0 & -1 \\
1 & 0 & 0 & -1 \\
0 & 1 & 0 & -1 \\
0 & 0 & 1 & -1 \end{array}\right)\; .$  To compute $\widehat{\rho}_1(b)$
we observe that in this case $k^{t-1} = 2^3 = 8$ and so we must find $1$
mod $f$, $x^8$mod $f$, $x^{16}$mod $f$ and $x^{24}$mod $f$.  Since
$\overline{x}^5 = \overline{1}$ in $\frac{\mathbb{F}_{19}[x]}{(f)}$ we
compute $1$ mod $f$, $x^3$mod $f$ and $x$ mod $f$ and $x^4$mod $f$.  We
get $1, x^3, x, -1 - x - x^2 - x^3$.  Hence
\[ \widehat{\rho}_1(b) = \left(\begin{array}{cccc}
1 & 0 & 0 & -1 \\
0 & 0 & 1 & -1 \\
0 & 0 & 0 & -1 \\
0 & 1 & 0 & -1 \end{array}\right)\; . \]

To compute a matrix representation $\widehat{\rho}_2$ over $\mathbb{F}_{19}$
equivalent to $\rho_2$ we must compute $z(x)$.  We have $\zeta^4 + \zeta^3 +
\zeta^2 + \zeta + 1 = 0$ and so $\zeta^2 + \zeta^{-2} + \zeta + \zeta^{-1} + 1
= 0$.  But $(\zeta + \zeta^{-1})^2 = \zeta^2 + \zeta^{-2} + 2$, whence
$\zeta^2 + \zeta^{-2} = (\zeta + \zeta^{-1})^2 - 2$.  Letting $\omega = \zeta
+ \zeta^{-1}$ we get $\omega^2 + \omega - 1 = 0$.  Hence $\omega =
\frac{-1\pm\sqrt{5}}{2} = \frac{-1\pm 9}{2}$.  Thus $\omega = 4$ or $\omega =
-5$ from which it follows that $x^4 + x^3 + x^2 + x + 1 = (x^2 - 4x + 1)(x^2 +
5x + 1)$. We may assume $\zeta$ is a root of $x^2 - 4x + 1$.  Adopting the
notation established above, we have $f = f_1f_2$, $\widehat{f}_1 = x^2 + 5x +
1$, $\widehat{f}_2 = x^2 - 4x + 1$.

We must find an element $z\in\mathbb{F}_{19}(\zeta)$ whose norm is $-1$.
Let $z = a_0 + a_1\zeta$.  We require $(a_0 + a_1\zeta)(a_0 +
a_1\zeta^{-1}) = -1$, i.e., $a^2_0 + a_0a_1(\zeta + \zeta^{-1}) + a^2_1
= -1$.  But $\zeta + \zeta^{-1} = 4$ and so $a^2_0 + 4a_0a_1 + a^2_1 + 1
= 0$.  Dividing by $a^2_1$ and setting $x = \frac{a_0}{a_1}$ we get $x^2
+ 4x + \frac{a^2_1 + 1}{a^2_1} = 0$.  Solving for $x$ we have $x =
-2\pm\frac{\sqrt{3a^2_1 - 1}}{a_1}$.  Let $a_1 = 2$.  Then $x =
-2\pm\frac{\sqrt{11}}{2} = -2\pm\frac{7}{2} = -2\pm 6$.  Hence $x = -8$
or $x = 4$.  Taking $\frac{a_0}{2} = -8$ we get $a_0 = 3$.  Hence $z = 3
+ 2\zeta$ and so $z_1(x) = 2x + 3$.  A routine computation establishes
that
\[ 1 = (2x - 8)(x^2 + 5x + 1) +(-2x + 9)(x^2 - 4x + 1)\; . \]
Hence $h_1(x) = (2x - 8)$ and $h_2(x) = -2x + 9$.  It follows that $z(x)
= (2x - 8)(x^2 + 5x + 1)(2x + 3) + (-2x + 9)(x^2 - 4x + 1)$.  Taking
$z(x)$ modulo $f$ (by abuse of notation) we have $z(x) = 4x^3 - x$.

Once again, $\widehat{\rho}_2(a) = \left(\begin{array}{cccc}
0 & 0 & 0 & -1 \\
1 & 0 & 0 & -1 \\
0 & 1 & 0 & -1 \\
0 & 0 & 1 & -1 \end{array}\right)\; . $ Now $b \cdot 1 = 4\overline{x}^3 -
\overline{x}$; $b\overline{x} = (4\overline{x}^3 -
\overline{x})\overline{x}^8$; $b\overline{x}^2 = (4\overline{x}^3 -
\overline{x})\overline{x}^{16}$; $b\overline{x}^3 = (4\overline{x}^3 -
\overline{x})\overline{x}^{24}$.  Reducing modulo $f$ we get $b\cdot 1 =
4\overline{x}^3 - \overline{x}$; $b\overline{x} = \overline{x}^3 +
\overline{x}^2 + 5\overline{x} + 1$; $b\overline{x}^2 = -4\overline{x}^3-
5\overline{x}^2 - 4\overline{x} - 4$; $b\overline{x}^3 = 4\overline{x}^2 - 1$.
Thus
\[ b\to \left(\begin{array}{rrrr}
0 & 1 & -4 & -1 \\
-1 & 5 & -4 & 0 \\
0 & 1 & -5 & 4 \\
4 & 1 & -4 & 0 \end{array}\right)\; . \]
\end{example}
{\bf Remark. } We have obtained necessary and sufficient conditions for the
realizability of $\rho^G$ over $\mathbb{F}_q$ in the case that $|b| = |k|_m$.
Furthermore, in the case when $|b|$ does not necessarily coincide with
$|k|_m$ and $q\equiv 1$ (mod $r$), we have given the necessary and
sufficient conditions for the complete realizability of $\rho^G$ over
$\mathbb{F}_q$.  There remains the problem of the mere realizability of
$\rho^G$ over $\mathbb{F}_q$ when $|b|\not= |k|_m$.  We observe in fact that
the condition $q\equiv k^j(m)$ for some $j$ is, even in this case, a necessary
and sufficient condition for the realizability of $\rho^G$ over $\mathbb{F}_q$
provided $(q,|G|) = 1$.  (Observe that throughout we are tacitly assuming that
$(m,q) = 1$ so that to require $(q,|G|) = 1$ we need only assume $(n,q) = 1$.)
Indeed the necessity follows from looking at the characteristic polynomial of
$\rho^G(a)$ while the sufficiency follows from the fact that $\rho^G$ is
completely realizable over $\mathbb{F}_q$ by Corollary~4.7. Alternatively one
could possibly obtain the sufficiency from two facts, namely:  (i) if
$q\equiv k^j(m)$, then $tr \rho^G(g)\in\mathbb{F}_q$ for all $g\in G$;  (ii)
the Schur index of a representation over a finite field is $1$ provided
$(q,|G|) = 1$ \cite[Theorem 24.10]{dor}.

Nevertheless, the whole thrust of this paper is to explicitly construct
the representations in question.  This could not be done by merely
appealing to the Schur index.

A number of new interesting problems arise from the paper.  We end the paper by
listing  a few of them

\begin{enumerate}
\item Examine the
case when $(n,q)\not= 1$.

\item Find a reciprocity law for other finite and also algebraic groups.

\item Extend reciprocity laws to cover fields which are not necessarily
finite.

\item  Find further applications to and relations with number-theoretic
reciprocity laws.

\end{enumerate}

\noindent
{\bf Acknowledgements.} We thank Franz Lemmermeyer for his interest in
this paper, and we thank the referee for his careful reading of our paper
and for his valuable suggestions which helped us to improve our
exposition.  In particular we added the referee's alternate, nice proof
of Lemma~2.3.

\end{document}